\documentclass[reqno, a4paper, 11pt]{amsart}

\usepackage{amssymb,amsmath,mathtools,amsthm,bbm}
\usepackage{enumerate}
\usepackage{xcolor}
\usepackage{mathrsfs}
\usepackage{enumitem}
\usepackage{color}

\usepackage{hyperref}  
\hypersetup{                    
	colorlinks=true,                
	breaklinks=true,                
	urlcolor= blue,                 
	linkcolor= black,                
	citecolor= magenta,                
	pdfstartview = FitH,	bookmarksopen = true               
}

\usepackage{txfonts}
\usepackage{newtxtext}
\usepackage{tikz}
\usepackage{caption}
\usepackage{subcaption}
\usepackage{accents}
\usepackage{cleveref}

\newcommand{\R}{\mathbb{R}}
\newcommand{\N}{\mathbb{N}}
\newcommand{\E}{\mathbb{E}}

\newcommand{\cW}{\mathcal{W}}

\newcommand{\cP}{\mathcal{P}}

\newcommand{\cN}{\mathcal{N}}
\newcommand{\cS}{{\mathcal S}}
\newcommand{\cO}{{\mathcal O}}
\newcommand{\tr}{{\rm tr}}
\newcommand{\BW}{{\bf bw}}

\renewcommand{\epsilon}{\varepsilon}

\theoremstyle{plain}
\newtheorem{prooff}{Proof}

\newtheorem{theorem}[prooff]{Theorem}

\newtheorem{corollary}[prooff]{Corollary}

\theoremstyle{definition}

\newtheorem{remark}[prooff]{Remark}
\newtheorem{definition}[prooff]{Definition}

\begin{document}

\title[Quadratic Wasserstein distance between Gaussian laws]{Quadratic Wasserstein distance between Gaussian laws revisited with correlation}

\author{Aur\'elien Alfonsi and Benjamin Jourdain
}
\thanks{CERMICS, ENPC, Institut Polytechnique de Paris, INRIA, Marne-la-Vall\'ee, France. E-mails: aurelien.alfonsi@enpc.fr, benjamin.jourdain@enpc.fr. This research benefited from the support of the chair Risques Financiers, Fondation du Risque.}
\maketitle
\begin{abstract} 
  In this note, we give a simple derivation of the formula obtained in \cite{Dowlan,GiSho,Olpuk} for the quadratic Wasserstein distance between two Gaussian distributions on $\R^d$ with respective covariance matrices $\Sigma_\mu$ and $\Sigma_\nu$. This derivation relies on the existence of an orthogonal matrix $O$ such that $O^*\Sigma_\mu O$ and $O^*\Sigma_\nu O$ share the same correlation matrix and on the simplicity of optimal couplings in the case with the same correlation matrix and therefore the same copula.

\end{abstract}
Let $\cP_2(\R^d)$ denote the set of probability measures on $\R^d$ with finite second order moment, $\mu,\nu\in\cP_2(\R^d)$ with expectations $m_\mu,m_\nu\in\R^d$ and covariance matrices $\Sigma_\mu,\Sigma_\nu$. The set of covariance matrices i.e. the subset of $\R^{d\times d}$ consisting of symmetric positive semi-definite matrices is denoted by $\cS^+_d$. The quadratic Wasserstein distance between $\mu$ and $\nu$ is defined by
$$\cW_2(\mu,\nu)=\left(\inf_{\pi\in\Pi(\mu,\nu)}\int_{\R^d\times \R^d}|x-y|^2\pi(dx,dy)\right)^{1/2},$$
where $\Pi(\mu,\nu)$ denotes the subset of $\cP_2(\R^d\times\R^d)$ consisting of probability measures with respective marginals $\mu$ and $\nu$. The covariance matrix of $\pi\in\Pi(\mu,\nu)$ writes $\left(\begin{array}{cc}\Sigma_\mu & \Theta_\pi\\\Theta_\pi^* & \Sigma_{\nu}\end{array}\right)\in\cS^+_{2d}$ for some $\Theta_\pi\in\R^{d\times d}$ and, by bias variance decomposition,
\begin{align}
  \int_{\R^d\times \R^d}&|x-y|^2\pi(dx,dy)=|m_\mu-m_\nu|^2+\int_{\R^d\times \R^d}|(x-m_\mu)-(y-m_\nu)|^2\pi(dx,dy)\notag\\&=|m_\mu-m_\nu|^2+\int_{\R^d\times \R^d}|x-m_\mu|^2+|y-m_\nu|^2-2(x-m_\mu)^*(y-m_\nu)\pi(dx,dy)\notag\\&=|m_\mu-m_\nu|^2+\tr(\Sigma_\mu+\Sigma_\nu-2\Theta_\pi),\label{w22}\end{align}
where the transpose of  $M\in\R^{d_1\times d_2}$ is denoted by $M^*\in\R^{d_2\times d_1}$ for $d_1,d_2\in \N^*$.
As a consequence, minimizing $\int_{\R^d\times \R^d}|x-y|^2\pi(dx,dy)$ with respect to $\pi\in\Pi(\mu,\nu)$ amounts to maximizing $\tr(\Theta_\pi)$. 
For each $\Theta\in\cS^+_d$ such that $\Gamma_\Theta:=\left(\begin{array}{cc}\Sigma_\mu & \Theta\\\Theta^* & \Sigma_{\nu}\end{array}\right)\in\cS_{2d}^+$, $\cN_{2d}\left(\left(\begin{array}{c}m_\mu\\ m_\nu\end{array}\right),\Gamma_\Theta\right)$ belongs to  $\Pi\left(\cN_d(m_\mu,\Sigma_\mu),\cN_d(m_\nu,\Sigma_\nu)\right)$ so that 
\begin{align*}
  \cW_2^2\left(\cN_d(m_\mu,\Sigma_\mu),\cN_d(m_\nu,\Sigma_\nu)\right)&=|m_\mu-m_\nu|^2+\tr(\Sigma_\mu+\Sigma_\nu)-2V(\Sigma_\mu,\Sigma_\nu)
\end{align*}
where
\begin{equation}
   V(\Sigma_\mu,\Sigma_\nu)=\sup_{\Theta\in\R^{d\times d}:\Gamma_\Theta\in\cS^+_{2d}}\tr(\Theta)\;\mbox{ with }\;\Gamma_\Theta:=\left(\begin{array}{cc}\Sigma_\mu & \Theta\\\Theta^* & \Sigma_{\nu}\end{array}\right)\label{defv}.
\end{equation}
Since $\{\Theta_\pi\in\R^{d\times d}:\pi\mbox{ coupling between }\mu\mbox{ and }\nu\}\subset\{\Theta\in\R^{d\times d}:\Gamma_\Theta\in\cS^+_d\}$, one also has
\begin{equation*}
   \cW_2^2(\mu,\nu)\ge |m_\mu-m_\nu|^2+\tr(\Sigma_\mu+\Sigma_\nu)-2V(\Sigma_\mu,\Sigma_\nu).
\end{equation*}The maximization problem defining $V(\Sigma_\mu,\Sigma_\nu)$ was solved by \cite{Dowlan,GiSho,Olpuk} who derived
\begin{equation}
   V(\Sigma_\mu,\Sigma_\nu)=\tr\left((\Sigma_\mu^{1/2}\Sigma_\nu\Sigma_\mu^{1/2})^{1/2}\right),\label{formv}
\end{equation}
where, for $\Sigma\in \cS^+_d$, $\Sigma^{1/2}$ denotes the symmetric square root of $\Sigma$ i.e. the unique element of $\cS^+_d$ such that $\Sigma^{1/2}\Sigma^{1/2}=\Sigma$. When $\Sigma$ is positive definite, we denote by $\Sigma^{-1/2}$ the matrix $(\Sigma^{-1})^{1/2}$ where $\Sigma^{-1}$ stands for the inverse of $\Sigma$. When $\Sigma_\mu$ is positive definite (resp. $\Sigma_\nu$ is positive definite), according to \cite{Olpuk}, the optimal $\Theta$ in \eqref{defv} is unique and equal to $\Sigma_\mu^{1/2}(\Sigma_\mu^{1/2}\Sigma_\nu\Sigma_\mu^{1/2})^{1/2}\Sigma_\mu^{-1/2}$ (resp. $\Sigma_\nu^{-1/2}(\Sigma_\nu^{1/2}\Sigma_\mu\Sigma_\nu^{1/2})^{1/2}\Sigma_\nu^{1/2}$). This is also a consequence of \cite{Gi}, according to which, when $\mu$ is absolutely continuous with respect to the Lebesgue measure then for each $\nu\in\cP_2(\R^d)$, there exists a unique $\cW_2$-optimal $\pi\in\Pi(\mu,\nu)$ and this coupling is given by a map $T$ (i.e. $\pi=(x,T(x))\#\mu(dx)$)  
and of Theorem 3.2.9 \cite{RaRu98} which ensures that the gradient of any convex function $u:\R^d\to\R$ is an optimal map. 
Indeed, the function $\R^d\ni x\mapsto T(x)=\Sigma_\mu^{-1/2}(\Sigma_\mu^{1/2}\Sigma_\nu\Sigma_\mu^{1/2})^{1/2}\Sigma_\mu^{-1/2}x$ is the gradient of the convex function $u(x)=\frac 12 x^*\Sigma_\mu^{-1/2}(\Sigma_\mu^{1/2}\Sigma_\nu\Sigma_\mu^{1/2})^{1/2}\Sigma_\mu^{-1/2}x$. Since $$\Sigma_\mu^{-1/2}(\Sigma_\mu^{1/2}\Sigma_\nu\Sigma_\mu^{1/2})^{1/2}\Sigma_\mu^{-1/2}\Sigma_\mu\Sigma_\mu^{-1/2}(\Sigma_\mu^{1/2}\Sigma_\nu\Sigma_\mu^{1/2})^{1/2}\Sigma_\mu^{-1/2}=\Sigma_\nu,$$
 we have $T\#\cN_d(0,\Sigma_\mu)=\cN_d(0,\Sigma_\nu)$ so that $T$ is the unique $\cW_2$-optimal transport map between $\cN_d(0,\Sigma_\mu)$ and $\cN_d(0,\Sigma_\nu)$. Moreover, when $X\sim\cN_d(0,\Sigma_\mu)$ then $\left(\begin{array}{c}
 X \\ T(X)
 \end{array}\right)\sim\cN_{2d}(0,\Gamma_\Theta)$ with $\Theta=\Sigma_\mu^{1/2}(\Sigma_\mu^{1/2}\Sigma_\nu\Sigma_\mu^{1/2})^{1/2}\Sigma_\mu^{-1/2}$. 
In \cite{CARTD}, optimality of the map $T$ is obtained by remarking that when $O\in\R^{d\times d}$ is orthogonal and such that $\Sigma_\mu^{-1/2}(\Sigma_\mu^{1/2}\Sigma_\nu\Sigma_\mu^{1/2})^{1/2}\Sigma_\mu^{-1/2}=ODO^*$ for some diagonal matrix $D\in\cS^+_d$, then the distributions of $O^*X$ and $O^*T(X)=DO^*X$ share the same copula (see \cite[Theorem 2.10]{CARTD}) in which case the optimal coupling amounts to map optimally each coordinates (see \cite[Theorem 2.9]{CARTD} and also \cite[Proposition 1.1]{AlJo14}). Since the Gaussian copulas are in one to one correspondence with correlation matrices, these two Gaussian random vectors also share the same correlation matrix, which can also be deduced from the equality $ODO^*\Sigma_\mu ODO^*=\Sigma_\nu$
 It follows from \eqref{formv} that
\begin{equation*}
   \cW_2^2\left(\cN_d(m_\mu,\Sigma_\mu),\cN_d(m_\nu,\Sigma_\nu)\right)=|m_\mu-m_\nu|^2+\BW^2(\Sigma_\mu,\Sigma_\nu)\label{w2gauss}
\end{equation*}
with the so-called Bures-Wasserstein distance on $\cS^+_d$ defined by
$$\BW(\Sigma_\mu,\Sigma_\nu)=\left(\tr\left(\Sigma_\mu+\Sigma_\nu-2(\Sigma_\mu^{1/2}\Sigma_\nu\Sigma_\mu^{1/2})^{1/2}\right)\right)^{1/2}.$$ 

To our knowledge, no optimal $\Theta$ was exhibited so far when both $\Sigma_\mu$ and $\Sigma_\nu$ are singular. In the present note, we give an independent and elementary proof of the equality \eqref{formv} and exhibit some $\Theta$ optimal in \eqref{defv} even when both $\Sigma_\mu$ and $\Sigma_\nu$ are singular. We rely on the existence of $O\in\R^{d\times d}$ orthogonal such that $O^*\Sigma_\mu O$ and $O^*\Sigma_\nu O$ share the same correlation matrix even in the singular case.

{\bf Notation :} 
Let $\cO_d$ denote the subset of $\R^{d\times d}$ consisting of orthogonal matrices and $I_d\in\R^{d\times d}$ denote the identity matrix. For $\lambda_\mu,\cdots,\lambda_d\in\R$, let ${\rm diag}(\lambda_\mu,\cdots,\lambda_d)$ denote the diagonal matrix with diagonal entries $\lambda_\mu,\cdots,\lambda_d$. For $\Sigma\in \cS^+_d$, let $\textup{dg}(\Sigma)={\rm diag}(\Sigma_{11},\cdots,\Sigma_{dd})$. Note that $\textup{dg}(\Sigma)^{1/2}={\rm diag}(\sqrt{\Sigma_{11}},\cdots,\sqrt{\Sigma_{dd}})$ and, when $\prod_{i=1}^d\Sigma_{ii}>0$, $\textup{dg}(\Sigma)^{-1/2}={\rm diag}(\Sigma_{11}^{-1/2},\cdots,\Sigma_{dd}^{-1/2})$. 
 We define the set of correlation matrices by $\mathfrak{C}_d= \{ C \in \cS^+_d ; \forall i, C_{ii}=1 \}$. \begin{definition}
  A correlation matrix associated with a matrix $\Sigma\in\cS^+_d$ is a matrix $C\in\mathfrak{C}_d$ such that
$\Sigma={\textup{dg}(\Sigma)}^{1/2}C{\textup{dg}(\Sigma)}^{1/2}$. We say that two matrices $\Sigma_1,\Sigma_2 \in \cS_d^+$ share the correlation matrix~$C \in \mathfrak{C}_d$ if we have $\Sigma_i={\textup{dg}(\Sigma_i)}^{1/2}C{\textup{dg}(\Sigma_i)}^{1/2}$ for $i\in \{1,2\}$. 
\end{definition}
The matrix $C$ with all diagonal entries equal to $1$ and all non-diagonal entries equal to those of
${\rm diag}\left(\frac{1_{\{\Sigma_{11}>0\}}}{\sqrt{\Sigma_{11}}},\cdots,\frac {1_{\{\Sigma_{dd}>0\}}}{\sqrt{\Sigma_{dd}}}\right)\Sigma{\rm diag}\left(\frac{1_{\{\Sigma_{11}>0\}}}{\sqrt{\Sigma_{11}}},\cdots,\frac {1_{\{\Sigma_{dd}>0\}}}{\sqrt{\Sigma_{dd}}}\right)$ is a correlation matrix associated with $\Sigma$. When $\prod_{i=1}^d\Sigma_{ii}>0$ (condition which holds when $\Sigma$ is non singular), then $C=\textup{dg}(\Sigma)^{-1/2}\Sigma\textup{dg}(\Sigma)^{-1/2}$ and this is the only correlation matrix associated with $\Sigma$.

\begin{theorem}\label{thw2cor}Let $\Sigma_\mu,\Sigma_\nu \in S_d^+$. Then, there exists $(O,C)\in \mathcal{O}_d\times \mathfrak{C}_d$ such that $O^*\Sigma_\mu O$ and $O^*\Sigma_\nu O$ share the correlation matrix $C$. 
  Besides, for any such couple $(O,C)$,  $\Theta=O{{\rm dg}(O^*\Sigma_\mu O)}^{1/2}C{{\rm dg}(O^*\Sigma_\nu O)}^{1/2}O^*$ is optimal in \eqref{defv} and we have 
 \begin{align*}
   V(\Sigma_\mu,\Sigma_\nu)&=\tr\left(O{{\rm dg}(O^*\Sigma_\mu O)}^{1/2}C{{\rm dg}(O^*\Sigma_\nu O)}^{1/2}O^*\right)= \sum_{i=1}^d \sqrt{(O^*\Sigma_\mu O)_{ii}(O^*\Sigma_\nu O)_{ii}}\\&=\tr\left(\left(\Sigma_\mu^{1/2}\Sigma_\nu \Sigma_\mu^{1/2}\right)^{1/2}\right).
  \end{align*} 
\end{theorem}
In view of the above discussion of the $\cW_2$ distance between Gaussian distributions and the equality $$\forall O\in\cO_d,\;\tr\left(\Sigma_\mu+\Sigma_\nu\right)=\tr\left(O^*\Sigma_\mu O+O^*\Sigma_\nu O\right)=\sum_{i=1}^d\left((O^*\Sigma_\mu O)_{ii}+(O^*\Sigma_\nu O)_{ii}\right),$$ we deduce the following corollary. 
\begin{corollary}\label{cor}Let $m_\mu,m_\nu\in\R^d$ and $\Sigma_\mu,\Sigma_\nu \in S_d^+$. Then, there exists $(O,C)\in \mathcal{O}_d\times \mathfrak{C}_d$ such that $O^*\Sigma_\mu O$ and $O^*\Sigma_\nu O$ share the correlation matrix $C$. For any such couple, $$\cN_{2d}\left(\left(\begin{array}{c}m_\mu
   \\m_\nu
  
\end{array}\right),\left(\begin{array}{cc}
   \Sigma_\mu &OD_\mu C D_\nu O^*\\O D_\nu CD_\mu O^*&\Sigma_\nu
                         \end{array}\right)\right)$$ with $(D_\mu,D_\nu)=({\rm dg}(O^*\Sigma_\mu O)^{1/2},{\rm dg}(O^*\Sigma_\nu O)^{1/2})$ is an optimal $\cW_2$ coupling between $\cN_{d}(m_\mu,\Sigma_\mu)$ and $\cN_d(m_\nu,\Sigma_\nu)$, and
                     \begin{align*}
                       \cW_2^2\left(\cN_{d}(m_\mu,\Sigma_\mu),\cN_d(m_\nu,\Sigma_\nu)\right)&=|m_\mu-m_\nu|^2+\BW^2(\Sigma_\mu,\Sigma_\nu)\mbox{ with }\\\BW^2(\Sigma_\mu,\Sigma_\nu)&=\sum_{i=1}^d \left(\sqrt{(O^*\Sigma_\mu O)_{ii}}-\sqrt{(O^*\Sigma_\nu O)_{ii}}\right)^2.\end{align*}
\end{corollary}
\begin{remark}
  \begin{itemize}
   \item The covariance matrix of the above optimal coupling also can be written $\left(\begin{array}{c}
   O{\rm dg}(O^*\Sigma_\mu O)^{1/2}\\O{{\rm dg}(O^*\Sigma_\nu O)}^{1/2}
   \end{array}\right)C\left(\begin{array}{c}
   O{{\rm dg}(O^*\Sigma_\mu O)}^{1/2}\\O{{\rm dg}(O^*\Sigma_\nu O)}^{1/2}\end{array}\right)^*$ so that when $Z\sim\cN_d(0,C)$, then $\left(\begin{array}{c}m_\mu
   \\m_\nu
  
\end{array}\right)+\left(\begin{array}{c}
   O{{\rm dg}(O^*\Sigma_\mu O)}^{1/2}\\O{{\rm dg}(O^*\Sigma_\nu O)}^{1/2}\end{array}\right)Z$ is distributed according this coupling. 
   \item When $X\sim \mathcal{N}_d(0,\Sigma_\mu)$, $Y\sim \mathcal{N}_d(0,\Sigma_\nu)$ are $\cW_2$-optimally coupled and $O\in\cO_d$, then, by the Cauchy-Schwarz inequality, 
   \begin{align*}
      \BW^2(\Sigma_\mu,\Sigma_\nu)&=\E[|X-Y|^2]=\E[|O^*X-O^*Y|^2]\\&=\sum_{i=1}^d\left(\E[(O^*X)_i^2]-2\E[(O^*X)_i(O^*Y)_i]+\E[(O^*Y)_i^2]\right)\\&\ge \sum_{i=1}^d\left(\sqrt{\E[(O^*X)_i^2]}-\sqrt{\E[(O^*Y)_i^2]}\right)^2\\&=\sum_{i=1}^d\left(\sqrt{(O^*\Sigma_\mu O)_{ii}}-\sqrt{(O^*\Sigma_\nu O)_{ii}}\right)^2
   \end{align*}
 and the equality holds if and only if for all $i\in\{1,\cdots,d\}$, there exists $\delta_i\in\R$ such that $(O^*X)_i=\delta_i (O^*Y)_i$ or $(O^*Y)_i=0$ a.s., i.e. if and only if  $O^*\Sigma_\mu O$ and $O^*\Sigma_\nu O$ share the same correlation matrix.  
 Therefore, the matrices $O\in\cO_d$ that  maximise the right-hand side and achieve the equality are precisely those  such that $O^*\Sigma_\mu O$ and $O^*\Sigma_\nu O$ share the same correlation matrix.    \item Assume that $\Sigma_1,\Sigma_2 \in S_d^+$ share the correlation matrix $C$. Then, for $\varepsilon \in(0,1)$, the symmetric matrices \begin{align*}
  \Sigma_i^{(\varepsilon)}=(\textup{dg}(\Sigma_i)+\varepsilon I_d)^{1/2}(\varepsilon I_d+(1-\varepsilon)C) (\textup{dg}(\Sigma_i)+\varepsilon I_d)^{1/2}, \ i=1,2,
\end{align*}  
are positive definite, share the correlation matrix $\varepsilon I_d+(1-\varepsilon)C$ and are such that $\Sigma_i^{(\varepsilon)}\to_{\varepsilon \to 0} \Sigma_i$.             
   \end{itemize}
   \label{rempr}\end{remark}\begin{proof}[Proof of Theorem \ref{thw2cor}]
   {\bf Step 1: existence of $O\in\cO_d$ such that $O^*\Sigma_\mu O$ and $O^*\Sigma_\nu O$ share the same correlation matrix.}
When $\Sigma_\mu$ is invertible, then the existence follows from the above discussion of the results in \cite{CARTD} and the orthogonal matrix~$O$ is obtained from the diagonalisation of $\Sigma_\mu^{-1/2}(\Sigma_\mu^{1/2}\Sigma_\nu\Sigma_\mu^{1/2})^{1/2}\Sigma_\mu^{-1/2}$.

When $\Sigma_\mu$ is singular, we consider a decreasing sequence $(\varepsilon_n)$ of real numbers such that $\varepsilon_n\to 0$ as $n\to\infty$. Then, there exists $C_n\in \mathfrak{C}_d$ and $O_n \in \cO_d$ such that 
\begin{align*}
   O_n^*(\Sigma_\mu+\varepsilon_nI_d)O_n&={\rm dg}(O_n^*(\Sigma_\mu+\varepsilon_nI_d)O_n)^{1/2}C_n{\rm dg}(O_n^*(\Sigma_\mu+\varepsilon_nI_d)O_n)^{1/2},\\O_n^*\Sigma_\nu O_n&={\rm dg}(O_n^*\Sigma_\nu O_n)^{1/2}C_n{\rm dg}(O_n^*\Sigma_\nu O_n)^{1/2}.
\end{align*}
By compactness of $\mathfrak{C}_d$ and $\cO_d$, we can extract a subsequence of $(C_n,O_n)$ that converges to  $(C,O)\in \mathfrak{C}_d\times \cO_d$, and we get \begin{align}
   O^* \Sigma_\mu O&={\rm dg}(O^*\Sigma_\mu O)^{1/2}C{\rm dg}(O^*\Sigma_\mu O)^{1/2},\;O^* \Sigma_\nu O={\rm dg}(O^*\Sigma_\nu O)^{1/2}C{\rm dg}(O^*\Sigma_\nu O)^{1/2} .\label{eqsamcor}\end{align}
 {\bf Step 2: derivation of $V(\Sigma_\mu,\Sigma_\nu)$ under \eqref{eqsamcor}.}
 Since the matrix $\Gamma_\Theta$ in \eqref{defv} belongs to $\cS^+_{2d}$ if and only if so does $\left(\begin{array}{cc}O^*\Sigma_\mu O& O^*\Theta O\\O^*\Theta^* O& O^*\Sigma_{\nu} O\end{array}\right)$, and by the cyclic property of the trace, $\tr(O^*\Theta O)=\tr(\Theta)$, we have $V(\Sigma_\mu,\Sigma_\nu)=V(O^*\Sigma_\mu O,O^*\Sigma_\nu O)$ and $\Theta$ is maximal for the optimization problem defining $V(\Sigma_\mu,\Sigma_\nu)$ in \eqref{defv} if and only if $O^*\Theta O$ is maximal for the one defining $V(O^*\Sigma_\mu O,O^*\Sigma_\nu O)$. \begin{equation}
   \mbox{For }\Sigma\in\cS^+_d,\;
   O^*\Sigma^{1/2}OO^*\Sigma^{1/2}O=O^*\Sigma O\mbox{ so that }O^*\Sigma^{1/2}O=(O^*\Sigma O)^{1/2}.\label{os}
 \end{equation} We deduce that \begin{align}\left((O^*\Sigma_\mu O)^{1/2}O^*\Sigma_\nu O(O^*\Sigma_\mu O)^{1/2}\right)^{1/2}=\left(O^*\Sigma_\mu^{1/2}\Sigma_\nu \Sigma_\mu ^{1/2}O\right)^{1/2}=O^*\left(\Sigma_\mu^{1/2}\Sigma_\nu \Sigma_\mu ^{1/2}\right)^{1/2}O,\label{raco}  
 \end{align}
                           so that, by the cyclic property of the trace,
                           \begin{equation*}
                              \tr\left(\left((O^*\Sigma_\mu O)^{1/2}O^*\Sigma_\nu O(O^*\Sigma_\mu O)^{1/2}\right)^{1/2}\right)=\tr\left(\left(\Sigma_\mu^{1/2}\Sigma_\nu \Sigma_\mu ^{1/2}\right)^{1/2}\right).
                           \end{equation*}
                           Hence, setting $\tilde\Sigma_\mu=O^*\Sigma_\mu O$ and $\tilde\Sigma_\nu=O^*\Sigma_\nu O$ for notational simplicity, it is enough to check that
                           \begin{align}
                              V(\tilde\Sigma_\mu,\tilde\Sigma_\nu)&=\sum_{i=1}^d\sqrt{(\tilde\Sigma_\mu)_{ii}(\tilde\Sigma_\nu)_{ii}}=\tr\left(\left(\tilde\Sigma_\mu^{1/2}\tilde\Sigma_\nu\tilde\Sigma_\mu^{1/2}\right)^{1/2}\right),\label{deosig}
                           \end{align}
with maximality of ${\rm dg}(\tilde \Sigma_\mu)^{1/2}C{\rm dg}(\tilde \Sigma_\nu)^{1/2}$ for the optimization problem defining $V(\tilde\Sigma_\mu,\tilde\Sigma_\nu)$. When $\Theta\in\R^{d\times d}$ is such that $\Gamma=\left(\begin{array}{cc} \tilde \Sigma_\mu & \Theta\\\Theta^* &  \tilde \Sigma_\nu \end{array}\right)$ belongs to $\cS^+_{2d}$, then $|\Theta_{ii}|\le \sqrt{(\tilde\Sigma_\mu)_{ii}(\tilde\Sigma_\nu)_{ii}}$ for each $i\in\{1,\cdots,d\}$ so that $V(\tilde\Sigma_\mu,\tilde\Sigma_\nu)\le \sum_{i=1}^d\sqrt{(\tilde\Sigma_\mu)_{ii}(\tilde\Sigma_\nu)_{ii}}$. On the other hand, this upper bound is attained for $\Theta={{\rm dg}(\tilde \Sigma_\mu)}^{1/2}C{{\rm dg}(\tilde \Sigma_\nu)}^{1/2}$ : we obviously have $\tr(\Theta)=  \sum_{i=1}^d\sqrt{(\tilde\Sigma_\mu)_{ii}(\tilde\Sigma_\nu)_{ii}}$ and $\Gamma=\left(\begin{array}{c}
    {{\rm dg}(\tilde \Sigma_\mu)}^{1/2}\\{{\rm dg}(\tilde \Sigma_\nu)}^{1/2}
    \end{array}\right)C\left(\begin{array}{c}
    {{\rm dg}(\tilde \Sigma_\mu)}^{1/2}\\{{\rm dg}(\tilde \Sigma_\nu)}^{1/2}
                             \end{array}\right)^*$, which shows that $\Gamma$ is well positive semidefinite. 
In view of the last point in Remark~\ref{rempr} and the continuity of $\cS^+_{d}\ni \Sigma\mapsto \Sigma^{1/2}\in\cS^+_{d}$, it is enough to check the second equality in~\eqref{deosig} when $\tilde \Sigma_\mu$ is non singular to cover the general case. We have  \begin{align*}
    (\tilde\Sigma_\mu^{1/2} {\rm dg}(\tilde\Sigma_\nu)^{1/2}{\rm dg}(\tilde\Sigma_\mu)^{-1/2}\tilde\Sigma_\mu^{1/2})^2&=\tilde\Sigma_\mu^{1/2} {\rm dg}(\tilde\Sigma_\nu)^{1/2}{\rm dg}(\tilde\Sigma_\mu)^{-1/2}\tilde\Sigma_\mu{\rm dg}(\tilde\Sigma_\mu)^{-1/2}{\rm dg}(\tilde\Sigma_\nu)^{1/2}\tilde\Sigma_\mu^{1/2} \notag\\
    &=\tilde\Sigma_\mu^{1/2} {\rm dg}(\tilde\Sigma_\nu)^{1/2}C {\rm dg}(\tilde\Sigma_\nu)^{1/2}\tilde\Sigma_\mu^{1/2} = \tilde\Sigma_\mu^{1/2} \tilde\Sigma_\nu \tilde\Sigma_\mu^{1/2},
   \end{align*}
     so that \begin{equation}
       \tilde\Sigma_\mu^{1/2} {\rm dg}(\tilde\Sigma_\nu)^{1/2}{\rm dg}(\tilde\Sigma_\mu)^{-1/2}\tilde\Sigma_\mu^{1/2}=\left(\tilde\Sigma_\mu^{1/2} \tilde\Sigma_\nu \tilde\Sigma_\mu^{1/2}\right)^{1/2}   \label{Map_ident}                              \end{equation}
     and, by the cyclic property of the trace, $${\rm tr}\left(\left(\tilde\Sigma_\mu^{1/2} \tilde\Sigma_\nu \tilde\Sigma_\mu^{1/2}\right)^{1/2}\right)={\rm tr}\left({\rm dg}(\tilde\Sigma_\nu)^{1/2}{\rm dg}(\tilde\Sigma_\mu)^{-1/2}\tilde\Sigma_\mu\right)=\sum_{i=1}^d\sqrt{(\tilde\Sigma_\mu)_{ii}(\tilde\Sigma_\nu)_{ii}}.$$\end{proof}
\begin{remark}\begin{itemize}
  \item When $\Sigma_\mu$ is non singular and \eqref{eqsamcor} holds, \eqref{os} and \eqref{Map_ident} imply that
    $$\Sigma_\mu^{1/2}O{\rm dg}(O^*\Sigma_\mu O)^{-1/2}{\rm dg}(O^*\Sigma_\nu O)^{1/2}O^*\Sigma_\mu^{1/2}=\left(\Sigma_\mu^{1/2}\Sigma_\nu\Sigma_\mu^{1/2}\right)^{1/2}.$$
 By left-multiplication by $\Sigma_\mu^{1/2}$, right-multiplication by $\Sigma_\mu^{-1/2}$ and using $\Sigma_\mu=O{\rm dg}(O^*\Sigma_\mu O)^{1/2} C{\rm dg}(O^*\Sigma_\mu O)^{1/2}O^*$, we deduce that
 $$O{\rm dg}(O^*\Sigma_\mu O)^{1/2} C{\rm dg}(O^*\Sigma_\nu O)^{1/2}O^*=\Sigma_\mu^{1/2}\left(\Sigma_\mu^{1/2}\Sigma_\nu\Sigma_\mu^{1/2}\right)^{1/2}\Sigma_\mu^{-1/2}.$$
 Moreover, maximality of $\Theta$ in \eqref{defv} implies that $(O^*\Theta O)_{ii}=\sqrt{(O^*\Sigma_\mu O)_{ii}(O^*\Sigma_\nu O)_{ii}}$ for $i\in\{1,\cdots,d\}$ so that when $\left(\begin{array}{c}X
   \\Y
 \end{array}\right)\sim\cN_{2d}\left(0,\left(\begin{array}{cc} O^*\Sigma_\mu O
 & O^*\Theta O  \\  O^*\Theta^* O & O^*\Sigma_\nu O  
 \end{array}\right)\right)$, then $Y_i=\sqrt{\frac{(O^*\Sigma_\nu O)_{ii}}{(O^*\Sigma_\mu O)_{ii}}}X_i$ for $i\in\{1,\cdots,d\}$. As a consequence for $i,j\in\{1,\cdots,d\}$, $(O^*\Theta O)_{ij}=\E[X_iY_j]=\sqrt{\frac{(O^*\Sigma_\nu O)_{jj}}{(O^*\Sigma_\mu O)_{jj}}}\E[X_iX_j]=\sqrt{(O^*\Sigma_\mu O)_{ii}}C_{ij}\sqrt{(O^*\Sigma_\nu O)_{jj}}$, so that $\Theta=O{\rm dg}(O^*\Sigma_\mu O)^{1/2} C{\rm dg}(O^*\Sigma_\nu O)^{1/2}O^*=\Sigma_\mu^{1/2}\left(\Sigma_\mu^{1/2}\Sigma_\nu\Sigma_\mu^{1/2}\right)^{1/2}\Sigma_\mu^{-1/2}$. 
In a symmetric way, when $\Sigma_\nu$ is non singular, the unique $\Theta$ maximal in \eqref{defv} is $\Sigma_\nu^{-1/2}(\Sigma_\nu^{1/2}\Sigma_\mu\Sigma_\nu^{1/2})^{1/2}\Sigma_\nu^{1/2}$. 
\item When both $\Sigma_\mu$ and $\Sigma_\nu$ are singular, uniqueness of $\Theta$ maximal in \eqref{defv} may fail as seen from the example $\Sigma_\mu=\left(\begin{array}{cc} 1 & 0
   \\ 0 & 0
\end{array}\right)$ and $\Sigma_\nu=\left(\begin{array}{cc} 0 & 0
   \\ 0 & 1
\end{array}\right)$, where for each $\rho\in [-1,1]$, $\Theta=\left(\begin{array}{cc} 0 & \rho
   \\ 0 & 0
                                                                    \end{array}\right)$ is maximal. Moreover non Gaussian couplings may also be $\cW_2$-optimal between Gaussian laws like the distribution of $(X,0,0,\varepsilon X)^*$ where $X\sim\cN_1(0,1)$ and $\varepsilon$ is an independent Rademacher random variable (${\mathbb P}(\varepsilon=-1)={\mathbb P}(\varepsilon=1)=\frac 12$)
                                                                  .
\end{itemize}
\end{remark}

\noindent{\bf Wasserstein barycenters.} We end up this note by making the connection with Wasserstein barycenters. Let $\Sigma_\mu,\Sigma_\nu \in \mathcal{S}_d^+$ and, by Theorem~\ref{thw2cor}, $(O,C)\in \mathcal{O}_d\times \mathfrak{C}_d$ such that $O^*\Sigma_\mu O$ and $O^*\Sigma_\nu O$ share the correlation matrix $C$. Then,  for each $\alpha\in[0,1]$, $$\eta_\alpha=\cN_{d}(\alpha m_\mu+(1-\alpha)m_\nu,O(\alpha D_\mu+(1-\alpha)D_\nu) C(\alpha D_\mu+(1-\alpha)D_\nu)O^*)$$ is a Wasserstein barycenter (see \cite{AgCa11}) of $\cN_{d}(m_\mu,\Sigma_\mu)$ and $\cN_d(m_\nu,\Sigma_\nu)$ with respective weights $\alpha$ and $1-\alpha$ i.e.
minimizes $$\alpha\cW_2^2(\eta,\cN_{d}(m_\mu,\Sigma_\mu))+(1-\alpha)\cW_2^2(\eta,\cN_d(m_\nu,\Sigma_\nu))$$ over $\eta\in\cP_2(\R^d)$. 
Indeed, we check easily by Corollary~\ref{cor} that $\cW_2^2(\eta,\cN_{d}(m_\mu,\Sigma_\mu))=(1-\alpha)^2\cW_2^2(\cN_{d}(m_\mu,\Sigma_\mu),\cN_{d}(m_\nu,\Sigma_\nu))$, $\cW_2^2(\eta,\cN_{d}(m_\nu,\Sigma_\nu))=\alpha^2\cW_2^2(\cN_{d}(m_\mu,\Sigma_\mu),\cN_{d}(m_\nu,\Sigma_\nu))$, so that  
$$\alpha\cW_2^2(\eta,\cN_{d}(m_\mu,\Sigma_\mu))+(1-\alpha)\cW_2^2(\eta,\cN_d(m_\nu,\Sigma_\nu))=\alpha(1-\alpha)\cW_2^2(\cN_{d}(m_\mu,\Sigma_\mu),\cN_d(m_\nu,\Sigma_\nu)).$$ On the other hand, for  $\eta(dz)k_z(dx)$ and $\eta(dz)\tilde k_z(dy)$ optimal $\cW_2$ couplings between $\eta$ and respectively $\cN_{d}(m_\mu,\Sigma_\mu)$ and $\cN_d(m_\nu,\Sigma_\nu)$,
\begin{align*}
\alpha&\cW_2^2(\eta,\cN_{d}(m_\mu,\Sigma_\mu))+(1-\alpha)\cW_2^2(\eta,\cN_d(m_\nu,\Sigma_\nu))\\&\ge \int_{\R^{3d}}\left(\alpha|x-z|^2+(1-\alpha)|y-z|^2\right)\eta(dz)k_z(dx)\tilde k_z(dy)\\&\ge \int_{\R^{3d}}\left(\alpha|x-(\alpha x+(1-\alpha)y)|^2+(1-\alpha)|y-(\alpha x+(1-\alpha)y)|^2\right)\eta(dz)k_z(dx)\tilde k_z(dy)\\&\ge \int_{\R^{3d}}\alpha(1-\alpha)|x-y|^2\eta(dz)k_z(dx)\tilde k_z(dy)\\&\ge \alpha(1-\alpha)\cW_2^2(\cN_{d}(m_\mu,\Sigma_\mu),\cN_d(m_\nu,\Sigma_\nu)).
\end{align*}

We now present the extension to $n\ge 2$ Gaussian distributions. Let $m_1,\dots,m_n\in \R^d$ and  $\Sigma_1,\dots,\Sigma_n\in \cS_d^+$ such that $O^*\Sigma_1O,\dots,O^*\Sigma_nO$ share the same correlation matrix~$C$ so that $O^*\Sigma_iO=D_i C D_i$ with $D_i$ a diagonal matrix for $i\in\{1,\cdots,n\}$. Then, 
$$\mathcal{N}_d\left( \sum_{i=1}^n p_i m_i, O\left( \sum_{i=1}^n p_i D_i \right)C\left( \sum_{i=1}^n p_i D_i \right)O^* \right)$$
is a Wasserstein barycenter of $(\mathcal{N}_d(m_i,\Sigma_i))_{1\le i \le n}$ with weights $(p_1,\cdots,p_n)\in[0,1]^n$ such that $\sum_{i=1}^n p_i=1$.
This follows from the general result that when $\mu^1,\cdots,\mu^n\in\cP_2(\R^d)$  share the same copula ${\bf C}$, then $$\eta_\star=\left(\sum_{i=1}^np_iF_{i1}^{-1}(u_1),\cdots,\sum_{i=1}^np_iF_{id}^{-1}(u_d)\right)\# {\bf C}(du_1,\cdots,du_d)$$ where $F_{ij}^{-1}$ denotes the quantile function of the $j$-th marginal $\mu^i_j$ of $\mu^i$  is a Wasserstein barycenter of $\mu^1,\cdots,\mu^n$ with weights $p^1,\cdots,p^n$. Indeed for $\eta\in\cP_2(\R^d)$ with marginals $\eta_1,\cdots,\eta_d$, using \cite{AgCa11} for the first equality,
\begin{align*}
\sum_{i=1}^n p_i\cW_2^2(\eta,\mu^i)&\ge \sum_{j=1}^d\sum_{i=1}^n p_i\cW_2^2(\eta_j,\mu^i_j)=\sum_{j=1}^d\sum_{i=1}^n p_i\int_0^1\left(F_{ij}^{-1}(u_j)-\sum_{k=1}^np_kF_{kj}^{-1}(u_j)\right)^2du_j\\
&=\sum_{i=1}^n p_i \int_{[0,1]^d}\sum_{j=1}^d\left(F_{ij}^{-1}(u_j)-\sum_{k=1}^np_kF_{kj}^{-1}(u_j)\right)^2 {\bf C}(du_1,\cdots,du_d)\\&\ge \sum_{i=1}^n p_i\cW_2^2(\eta_\star,\mu^i).
\end{align*}

\bibliographystyle{abbrv}
\bibliography{joint_biblio}
\end{document}